\documentclass[
11pt,%
tightenlines,%
twoside,%
onecolumn,%
nofloats,%
nobibnotes,%
nofootinbib,%
superscriptaddress,%
noshowpacs,%
centertags]%
{revtex4-2}
\usepackage{ljm}

\DeclareMathOperator{\supp}{supp}
\usepackage{wasysym}

\newcommand{\RR}{\mathbb{R}} 
\newcommand{\CC}{\mathbb{C}} 
 
\newcommand{\DD}{\mathbb{D}} 

\theoremstyle{remark}
\newtheorem{remark}{\bf Remark}
\begin{document}
\titlerunning{The Harnack distance and estimates of subharmonic functions} 
\authorrunning{Khabibullin B.N.} 

\title{The Harnack distance and  estimates of subharmonic functions  from below on the unit disc}

\author{\firstname{B.N.}~\surname{Khabibullin}}
\email[E-mail: ]{khabib-bulat@mail.ru} \affiliation{Institute of Mathematics with Computing Centre - Subdivision of the Ufa Federal Research Centre of the Russian Academy of Sciences,
112, Chernyshevsky str., Ufa, Russian Federation, 450008}



\begin{abstract} 
General estimates from below of holomorphic and subharmonic functions play one of the key roles in the theory of growth of holomorphic and subharmonic functions and in general in the theory of potential. At the same time, the most diverse technical apparatus was often used for this. In one of our recent papers, we showed in fairly general cases that in fact such a wide set of tools is not required. The classic Harnack distance is almost always enough. In this article, we will illustrate how this  works in the cases of the unit  disc in the complex plane.
\end{abstract}

\maketitle


\section{INTRODUCTION. Previous general results}

A one-point set is denoted by \{x\} or $x$.
${\mathbb N}$:=\{1,2, \dots\} and ${\mathbb R}$ 
are  sets {\it natural\/} and {\it real\/} numbers;  
${\mathbb R}^+:=\bigl\{x\in {\mathbb R}\bigm| x\geq 0\bigr\}$.
Throughout the paper, ${d}\in {\mathbb N}$ 
denotes the dimension of the Euclidean space ${\mathbb R}^{d}$  equipped with the Euclidean norm |x| :=
$|x|:=\sqrt{x_1^2+\dots +x_{d}^2}$ 
for $x:=(x_1,\dots ,x_{d})\in {\mathbb R}^{d}$,  the Euclidean distance dist(·, ·), and the Alexandroff
one-point compactification
${\mathbb R}_{\infty}^{d}:={\mathbb R}^{d}\cup \infty$. 
For $x \in \mathbb R \cup{±\infty}$ we denote by $x^+$ and
$x^-$ the positive and negative parts of x respectively. 

For a set  $S\subset {\mathbb R}^{d}$
 we denote by  $\complement S:={\mathbb R}_{\infty}^{d}\setminus S$, $\overline S$, $S^\circ$ 
 and   $\partial S$   the complement, closure, interior, and boundary of $S$ in $R^{d}_{\infty}$.
Throughout the paper, $D\neq \emptyset$  denotes a domain in ${\mathbb R}^{d}$.

For $r\in \overline{\mathbb R}^+$ and   $x\in {\mathbb R}^{d}$ we set 
$B_x(r):=\bigl\{y\in {\mathbb R}^{d}\bigm| |y-x|<r\bigr\}$,   
$\overline  B_x(r):=\bigl\{y\in {\mathbb R}_{\infty}^{d}\bigm| |y-x|\leq r\bigr\}$ 
and  $\partial \overline B_x(r):=\overline B_x(r)\setminus  B_x(r)$.  The area of the unit sphere
$\partial B_0(1) \subset  \mathbb R^{d}$ is expressed in terms of the Euler $\Gamma$-function as follows:
\begin{equation}\label{{kK}s}
s_{d-1}=\frac{2\pi^{d/2}}{\Gamma (d/2)}, 
\quad s_{ 0}=2, \;
s_{ 1}=2\pi, \; s_{ 2}=4\pi,\;  s_{ 3}=\pi^2, \dots  . 
\end{equation}

The increasing function
\begin{equation}\label{kKd-2}
\Bbbk_{d-2} \colon  t\underset{0<t\in {\mathbb R}^+}{\longmapsto} \begin{cases}
t &\text{\it  when ${d=1}$},\\
\ln t  &\text{\it  when ${d=2}$},\\
 -\dfrac{1}{t^{d-2}} &\text{\it when ${d>2}$,} 
\end{cases} 
\qquad \Bbbk_{d-2} (0):=\lim_{0<t\to 0} \Bbbk_{d-2} (t)\in \overline {\mathbb R},
\end{equation}
depending on $d$ plays a key role in the theory of potentials and subharmonic functions because
for every fixed $x \in R^{d}$ the function
\begin{equation}\label{kd0}
y\underset{y \in {\mathbb R}^{d}}{\longmapsto}
\frac{1}{s_{d-1}\widehat{d}}\,\Bbbk_{d-2}\bigl(|y-x|\bigr), \qquad
\widehat{d}:=\max\{{1,d-2}\}=1+({d-3})^+\in \mathbb N,
\end{equation}
is the fundamental solution to the Laplace
equation ${\bigtriangleup} u=\delta_x$  in $\mathbb R^{d}$, where ${\bigtriangleup }$ is the Laplace operator acting in the sense of the theory of
distributions and $\delta_x$  is the Dirac probability measure with $\supp \delta_x =\{x\}$ at a single point x.
We often omit the subscript $d- 2$ and write $\Bbbk $ instead of $\Bbbk_{d-2}$ if the value of $d$ is clear from thу context.

We denote by  ${\sf har}(S)$ and  ${\sf sbh}(S)$  the classes of restrictions of all harmonic (locally affine for
$d = 1$) functions on $S \subset \RR^{d}$ and subharmonic (locally convex for $d = 1$) functions on some open
neighborhoods of S respectively \cite{Rans}, \cite{HK}, \cite{Brelot}.


\begin{definition}[{\cite[1.3]{Rans}, \cite{Kohn}}]\label{dH}
For a domain  $D\subset {\mathbb R}^{d}$ and the cone ${\sf har}^+(D)\subset {\sf har}(D)$ of all
positive harmonic functions in $D$ the Harnack distance between points $x\in D$ and  $y\in D$ 
on  $D$
defined by
\begin{equation}\label{hard}
{\sf dist}_{{\sf har}}^D (x,y):=\inf\Bigl\{ t\in {\mathbb R}^+\Bigm|
\frac{1}{t}  h(y)\leq h(x) \leq t h(y)\quad   
\text{for all\;  $h\in {\sf har}^+(D)$}\Bigr\}\geq 1.
\end{equation}
\end{definition}

The main properties of the Harnack distance can be found in \cite[1.3]{Rans}, \cite{Kohn}. In particular, in addition
to the symmetry property, the following multiplicative triangle inequality holds:
\begin{equation}\label{pmt}
{\sf dist}_{{\sf har}}^D(x,y)\leq {\sf dist}_{{\sf har}}^D(x,o)\cdot {\sf dist}_{{\sf har}}^D(o,y)\text{ \it for each $x,o,y\in D$}.
\end{equation} 
For the Harnack distance we have the subordination principle \cite[Theorem 1.3.6]{Rans}, \cite[\S~3, 
Theorem 3.3]{Kohn} which will be
used only in the particular case\cite[Corollary  1.3.7]{Rans}, \cite[\S~3, 
Theorem 3.2]{Kohn}: 
\begin{equation}\label{pr}
{\sf dist}_{{\sf har}}^G\leq {\sf dist}_{{\sf har}}^{D} \quad \text{\it  in $D\subset G$}.
\end{equation} 

In the case of a ball, the Harnack distance to the ball center can be explicitly written as \cite[\S~3, application]{Kohn}:
\begin{equation}\label{DB}
{\sf dist}_{{\sf har}}^{B_{o}(r)}(o, x)=\frac{\bigl(r+|x-o|\bigr)r^{{ d-2}}}{\bigl(r-|x-o|\bigr)^{{ d}-1}},
\quad  x\in B_{o}(r).
\end{equation}
The Harnack distance increases with increasing the distance $|x-x_0|$ from a point $x$ to the center
$x_0$ of the ball $B_{x_0}(R)$.

The class ${\sf sbh}(S)$  contains also a function, denoted by $-\infty$, that coincides with the restriction
on S of the function equal identically to ${\sf sbh}(S)$  in some open neighborhood of the set $S$. The
equality relation = for a pair of functions on ${\sf sbh}(S)$ means that there exists an open set in $\RR^d$
containing $S$ where both functions are defined and pointwise coincide. If $O \subset \RR^d$ is an open set
and  $u\in {\sf sbh}(O)$  is a subharmonic function such that $u\not\equiv-\infty$ on each connected component of
$O$, then the Riesz measure of u is uniquely defined by
\begin{equation}\label{df:cm}
\varDelta_u\overset{\eqref{kd0}}{:=} \frac{1}{s_{ d-1}\widehat{ d}} {\bigtriangleup}  u, 
\end{equation}
It is the Radon measure on $O$, i.e., the Borel measure that is finite on compact sets of $O$. If
$u\equiv -\infty$ on some connected component of $O$, then from the definition of the Riesz measure it
follows that $\varDelta_u$ takes the value $+\infty$ on each subset of this connected component. This definition
is extended to a subharmonic function   $u\in {\sf sbh}(S)$, on the Borel set $S$, regarded as the restriction
on $S$ of the Riesz measure of the subharmonic function $u$ in some open neighborhood of $S$.
For the Borel measure $\mu$  on 
 на ${\mathbb R}^{ d}$ и точки $x\in {\mathbb R}^{ d}$
 we denote by \begin{equation}\label{muyr} 
\mu_x^{\text{\tiny \rm rad}}(t)\underset{t\in {\mathbb R}^+}{:=}\mu\bigl(\overline B_x(t) \bigr)\in \overline {\mathbb R}^+ 
\end{equation}
 the radial counting function of the measure 
$\mu$   with center  $x\in {\mathbb R}^{ d}$, i.e., it is a function that
is continuous from the right on $\RR^+$ and acts from $\RR^+$ to $\RR^+$. Furthermore, 
for $\widehat{ d}\overset{\eqref{kd0}}{=}1+({ d}-3)^+$ 
\begin{equation}\label{muyrN} 
{\sf N}_x^{\mu} (r)\underset{r\in \overline {\mathbb R}^+}{:=}\widehat{ d}\int_0^r\frac{\mu_x^{\text{\tiny \rm rad}}(t)}{t^{{ d}-1}} 
{\,{\mathrm d}} t\in \overline {\mathbb R}^+ 
\end{equation}
is the radial averaged or integrated counting function of the measure $\mu$ with center $x\in \RR^d$. 
It  is increasing and continuous on $\RR+$ as a function from $\RR^+$ to $\RR^+$.

We denote by 
$\diameter\!S:=\sup\limits_{x,y\in S}|x-y|$
 the Euclidean diameter of a set $S \subset  \RR^d$.

\subsection{\bf The main results on estimates from below}
 The first result concerns lower estimates for the values of functions at a single point of a domain.

\begin{theorem}[{\cite[Theorem 1.1]{KhaTal22}}]\label{th1_1}
Let $D$  be a bounded domain in  $ {\mathbb R}^{d}$  with a fixed point $o\in D$, and let $u$ be
a subharmonic function on $D$ with $u(o) = 0$ and the Riesz measure  $\varDelta_u$. Then the restriction
$\varDelta\overset{\eqref{df:cm}}{:=}\varDelta_u\!\bigm|_{\overline {D}}$
of the Riesz measure $\varDelta_u$ on $\overline {D}$  satisfies the inequality
\begin{equation}\label{infuI20}
u(x)\geq -\bigl({\sf dist}_{{\sf har}}^{D}(o,x) -1\bigr)\sup_{\partial D} u
-{\sf N}_x^{\varDelta}(\diameter\!D)
\quad \text{for each point   $x\in D$}.
\end{equation}        

If for $u\in {\sf sbh}(\overline D)$ we are given a domain $G\subset {\mathbb R}^{d}$such that
\begin{equation}\label{oBDG+}
o\in  D\subset \overline D\subset G\subset \overline G\subset {\mathbb R}^{ d},\quad 
u\in {\sf sbh}(\overline G),\quad u(o)=0,
\end{equation}
then for any $r_x\in (0, \diameter\!D]$ 
\begin{subequations}\label{rd}
\begin{gather}
\hspace{-1mm}u(x)\geq -\bigl({\sf dist}_{{\sf har}}^{D}(o,x) -1\bigr)\sup_{\partial D} u- 
\frac{\Bbbk(\diameter\!D)-\Bbbk(r_x)}{\Bbbk(R+l_D^G)-\Bbbk(R)}\sup\limits_{y\in \partial D}{\sf dist}_{{\sf har}}^{G\!\setminus\!o}\bigl(y,\partial B_o(R)\bigr)\sup_{\partial G} u
 -{\sf N}_x^{\varDelta}(r_x),
\tag{\ref{rd}u}\label{infuI2}
\\
\intertext{where $\sup\limits_{\partial D} u$ can be replaced with $\sup\limits_{\partial G} u$ by the maximum principle,}
R:={\sf dist}(o,\partial D)
\tag{\ref{rd}R}\label{rdr}
\\
l_D^G:={\sf dist}(D, \complement G).
\tag{\ref{rd}d}\label{rdd}
\end{gather}
\end{subequations}
For $u(x)\neq -\infty$ the right-hand sides of 
\eqref{infuI20} and   \eqref{infuI2}  are finite.
\end{theorem}

\begin{remark} The condition $u(o) = 0 $ for $u(o)\not\equiv  -\infty$ can be omitted 
if $u(x)$ is replaced with the difference $u(x)-u(o)$ on the left-hand sides of \eqref{infuI20} and \eqref{infuI2}; 
moreover $\sup\limits_D u$ is replaced 
with  $\sup\limits_D u-u(o)$  on the right-hand side of  \eqref{infuI2}, and 
$\sup\limits_G u$  is replaced with $\sup\limits_G u-u(o)$ on
the right-hand side of \eqref{infuI2}. Under the same replacements, the inequalities \eqref{infuI20} and  \eqref{infuI2} also
hold for $u(o)=-\infty$ in the case $u(x)\neq -\infty$ since both sides of \eqref{infuI20} и \eqref{infuI2} are well defined,
whereas the left-hand side is $+\infty$.
\end{remark}

Regarding lower estimates for a subharmonic function $u\in  {\sf sbh}(\overline D)$ on some set, it is rea-
sonable, as a rule, to get an estimate outside some exceptional small set without mentioning
the Riesz measure of a function u or its restriction. We use variant of such
lower estimates based on the so-called normal point method coming back to the Cartan lower
estimate for the modulus of a polynomial. For this purpose we need the following definition.

\begin{definition}[{\cite{Carleson}, \cite{Federer}, \cite{Rodgers}, \cite{HedbergAdams}, \cite{Eid07}, \cite{VolEid13}}]\label{defH}
For  $r\in \overline {\mathbb R}^+\setminus 0$  and  $h\colon [0,r)\to {\mathbb R}^+$
the set function
\begin{equation}\label{mr}
{\mathfrak m}_h^{\text{\tiny $r$}}\colon S\underset{S\subset {\mathbb R}^{ d}}{\longmapsto}  \inf \Biggl\{\sum_{j\in N} h(r_j)\biggm| N\subset {\mathbb N},\,  S\subset \bigcup_{j\in N} 
\overline B_{x_j}(r_j), \, x_j\in {\mathbb R}^{ d}, \, r_j\underset{j\in N}{\leq} r\Biggr\} \in \overline {\mathbb R}^+
\end{equation}
is called the {\it Hausdorff $h$-content of radius $r$.\/} For every  $S\subset {\mathbb R}^{ d}$ the values    ${\mathfrak m}_h^{\text{\tiny $r$}}(S)$  is decreasing
in $r$ and has the limit 
\begin{equation}\label{hH}
{\mathfrak m}_h^{\text{\tiny $0$}}(S):=\lim_{0<r\to 0} {\mathfrak m}_h^{\text{\tiny $r$}}(S)
\geq {\mathfrak m}_h^{\text{\tiny $r$}}(S)\geq {\mathfrak m}_h^{\text{\tiny $\infty$}}(S)
 \quad \text{\it for all $S\subset {\mathbb R}^{d}$}. 
\end{equation}
For $h(0)=0$ all contens  ${\mathfrak m}_h^{\text{\tiny $r$}}$ are exterior measures, and  ${\mathfrak m}_h^{\text{\tiny $0$}}$ defines the Hausdorff $h$-measure ${\mathfrak m}_h^{\text{\tiny $0$}}$ which is a regular Borel measure. The most commonly used functions are power functions  $h_p$ of
degree $p\in {\mathbb R}^+$  with a normalizing factor of the form
\begin{equation}\label{hd}
h_p\colon t\underset{t\in {\mathbb R}^+}{\longmapsto} 
c_pt^p, \quad\text{where } 
c_p:=\dfrac{\pi^{p/2}}{\Gamma(p/2+1)}, 
\end{equation}
The Hausdorff $h_p$-content of radius $r$ и  $h_p$-measure are referred to as the $p$-dimensional content of radius r and the Hausdorff measure respectively and denoted by
\begin{equation}\label{p-m}
p\text{-}{\mathfrak m}^{\text{\tiny $r$}}:={\mathfrak m}_{h_p}^{\text{\tiny $r$}}, \quad  
p\text{-}{\mathfrak m}^{\text{\tiny $0$}}:={\mathfrak m}_{h_p}^{\text{\tiny $0$}}. 
\end{equation}
\end{definition}

\begin{corollary}[{\cite[Corollary 1.1]{KhaTal22}}]\label{cor1}
Let the assumptions of Theorem {\rm \ref{th1_1}}  hold, and let
$S\subset \overline S\subset D$, $0<r\leq \diameter\!D$.
Then for any function $h\colon [0,r]\to {\mathbb R}^+$ such that  $h(0)=0$ and
\begin{equation}\label{N0h}
N_0^h(r):=\widehat{ d}\int_0^r\frac{h(s)}{s^{{ d}-1}}{\,{\mathrm d}} s<+\infty
\end{equation}
there is a set  $E\subset S$ such that
\begin{subequations}\label{eEr}
\begin{align}
 \inf_{x\in S\setminus E} u(x)&\geq -\biggl(
\sup_{x\in S}{\sf dist}_{{\sf har}}^{D}(o,x) -1
\notag
\\
&+\frac{\Bbbk(\diameter\!D)-\Bbbk(r)}{\Bbbk(R+l_D^G)-\Bbbk(R)}\sup\limits_{x\in \partial D}{\sf dist}_{{\sf har}}^{G\!\setminus\!o}\bigl(x,\partial B_o(R)\bigr)+N_0^h(r)\biggr)
\sup_{\partial G} u
\tag{\ref{eEr}u}\label{{eEr}u}
\\
\intertext{and the following estimate from above holds:}
\mathfrak m_h^r(E)&\leq 
\frac{5^{ d}}{\Bbbk(R+l_D^G)-\Bbbk(R)}\sup\limits_{x\in \partial D}{\sf dist}_{{\sf har}}^{G\!\setminus\!o}\bigl(x,\partial B_o(R)\bigr).
\tag{\ref{eEr}E}\label{{eEr}E}
\end{align}
\end{subequations}

In particular, in the case of a power function \eqref{hd} with degree  $p\in ({ d}-2,{ d}]$ multiplied by $B\in {\mathbb R}^+\setminus 0$, the quantity $N_0^h(r)$ on the right-hand side of \eqref{{eEr}u} takes the form
 \begin{equation}\label{N0}
N_0^h(r)\overset{\eqref{N0h}}{=}Bc_p\widehat{ d}\int_0^rt^{p-{ d}+1}{\,{\mathrm d}} t
=B\frac{c_p\widehat{ d}}{p-({ d}-2)}r^{p-({ d}-2)},
\end{equation}
and the $p$-dimensional content \eqref{p-m} of radius $r$ of the exceptional set $E$ satisfies the estimate
\begin{equation}\label{pBB}
p\text{-}{\mathfrak m}^{\text{\tiny $r$}}(E)\overset{\eqref{{eEr}E}}{\leq}
\frac{1}{B}\cdot\frac{5^{ d}}{\Bbbk(R+l_D^G)-\Bbbk(R)}\sup\limits_{x\in \partial D}{\sf dist}_{{\sf har}}^{G\!\setminus\!o}\bigl(x,\partial B_o(R)\bigr).
\end{equation}

\end{corollary}

\section{Case of disc in the complex plane $\CC$}
We denote by $\DD:=\bigl\{z\in \CC\bigm| |z|<1 \bigr\}$  \textit{the unit disc\/} in the \textit{complex plane\/} $\CC$, i.e., $d:=2$. Within the framework of the theorem \ref{th1_1} and corollary \ref{cor1}, we assume 
\begin{equation}\label{GD}
 o:=0, \quad G:=\DD, \quad D:=R\DD=B_o(R)\quad\text{where $R\in (0,1)$},
\end{equation}
 i.e., the equality \eqref{rdr} takes place and  
\begin{equation}\label{rdd1}
l_D^G\overset{\eqref{rdd}}{=}{\sf dist}(D, \complement G)={\sf dist}(R\DD, \complement \DD)=1-R\in (0,1).
\end{equation}
\begin{theorem}\label{th2} Consider a subset 
\begin{equation}\label{Ss0}
S\subset \overline S\subset R\DD,  \quad 
s_0:=\sup \bigl\{|s| \bigm| s\in S\bigr\}<R <1.
\end{equation}
Let $u$ be  a subharmonic function on $\overline \DD$ with $u(0) = 0$,
and
\begin{equation}\label{rR}
 0<r\leq 2R\in (0,2).  
\end{equation}
for any function $h\colon [0,r]\to {\mathbb R}^+$, $0<r\leq 2R$, such that  $h(0)=0$ and 
\begin{equation}
N_0^h(r)\overset{\eqref{N0h}}{:=}\int_0^r\frac{h(s)}{s}{\,{\mathrm d}} s<+\infty
\end{equation}
there is a set  $E\subset S$ such that
\begin{equation}\label{frb}
 \inf_{z\in S\setminus E} u(x)
\geq -\biggl(
\frac{2s_0}{R-s_0}+\frac{\ln (2R/r)}{-\ln R}
\Bigl(\frac{1+R}{1-R}\Bigr)^2+N_0^h(r)\biggr)
\sup_{\partial \DD} u,
\end{equation}
where the exceptional set $E$ is bounded in the sense that
\begin{equation}\label{NB}
\mathfrak m_h^r(E)\leq 
\frac{5^{ 2}}{-\ln R}\Bigl(\frac{1+R}{1-R}\Bigr)^2\leq \frac{100}{(1-R)\ln\frac1{R}}
\end{equation}
In particular, in the case of a power function \eqref{hd} with degree  $p\in (0,2]$ multiplied by $B\in {\mathbb R}^+\setminus 0$, the quantity $N_0^h(r)$ on the right-hand side of \eqref{{eEr}u} takes the form
 \begin{equation}\label{Nsh}
N_0^h(r):=B\frac{\pi}{p}r^p,
\end{equation}
and the $p$-dimensional content \eqref{p-m} of radius $r$ of the exceptional set $E$ satisfies the estimate
\begin{equation}\label{pB}
p\text{-}{\mathfrak m}^{\text{\tiny $r$}}(E)\overset{\eqref{{eEr}E}}{\leq}
\frac{1}{B}\cdot\frac{5^{2}}{-\ln R}\Bigl(\frac{1+R}{1-R}\Bigr)^2\leq \frac{100}{B(1-R)\ln \frac1{R}}.
\end{equation}
Thus, the choice of a sufficiently large value 
$B$ allows you to significantly reduce the exceptional set of $E$, but at the same time increases {Nsh} by a multiplier of $B$, i.e. increases by the same multiplier $B$ the last term in the bracket to the right of the estimate \eqref{frb} from below.

\end{theorem}
\begin{proof}
Iт this case for each point $z\in S\subset R\DD$ by \eqref{DB} we  obtain 
\begin{equation}\label{DB1}
{\sf dist}_{{\sf har}}^{R\DD}(0, z)-1
\overset{\eqref{DB}}{=}\frac{R+|z|}{R-|z|}-1=
\frac{2|z|}{R-|z|},
\quad  z\in S\subset \overline S\subset R\DD.
\end{equation}
For an intermediate fraction-multiplier \eqref{{eEr}u}  with the 
  function $\Bbbk=\ln $ from \eqref{kKd-2} for $d=2$ we have
\begin{equation}\label{Bbbk}
\frac{\Bbbk(\diameter\!D)-\Bbbk(r)}{\Bbbk(R+l_D^G)-\Bbbk(R)}
\overset{\eqref{GD}}{=}\frac{\ln \diameter\!(R\DD)-\ln r}{\ln(R+l_D^G)-\ln R}
\overset{\eqref{rdd1}}{=}\frac{\ln \diameter\!(R\DD)-\ln r}{\ln(R+(1-R))-\ln R} 
=\frac{\ln (2R/r)}{-\ln R}\geq 0
\end{equation}
since according to the conditions $0<r<\diameter\!D=2R$ and $R\overset{\eqref{GD}}{\in} (0,1)$

Let's move on to the assessment of Harnack distance in  \eqref{{eEr}u}: 
\begin{equation}\label{mainpr}
{\sf dist}_{{\sf har}}^{G\!\setminus\!o}\bigl(w,\partial B_o(R)\bigr)
={\sf dist}_{{\sf har}}^{\DD\!\setminus\!0}\bigl(w,\partial (R\DD)\bigr) 
\text{where $w\in \partial D=\partial (R\DD)$, i.e.  always $|w|=R$.}
\end{equation}
According to \eqref{mainpr}, we need to estimate  from above  the Harnack distance in the punctured unit disc $\DD\setminus 0$ between two arbitrary points $w$ и $w_0$ with identical modules $|w|=|w_0|=R$.
The principle of subordination \eqref{pr} will not help here, since  $\DD\setminus 0\subset  \DD$ and not vice versa.

Consider a positive  harmonic function $h$ on $\overline \DD\setminus 0$. By  \cite[B\"ocher's Theorem]{ABR}  there a constant $b\in \RR^+$ such that we have   the following  representation  
\begin{equation}\label{hhh}
h(w)=b\ln \frac{1}{|w|}+v(w), \quad b\in \RR^+, \quad\text{for all $w\in \overline \DD$} 
\end{equation}
where $v$ is a harmonic function on $\DD$   continuous on $\partial \DD$.  By construction  $v\geq 0$ on  
$\partial \DD$ since $b\ln 1/|w|=0$ when $|w|=1$. Therefore, according to the principle of maximum the function $v$ is \textit{positive} and harmonic  on $\DD$. By the Harnack  distance for $\DD$ we obtain 
\begin{equation}\label{1i}
b\ln \frac{1}{|w|}+v(w)\leq b\ln \frac{1}{|w|}+{\sf dist}_{{\sf har}}^{\DD}(w,w_0)v(w_0)
\leq {\sf dist}_{{\sf har}}^{\DD}(w, w_0) \Bigl( b\ln \frac{1}{|w_0|}+v(w_0)\Bigr)
\end{equation}
since ${\sf dist}_{{\sf har}}^{\DD}(w, w_0)\geq 1$ and $|w|=|w_0|=|R|$. By 
the  multiplicative triangle and the Harnack distance  in the case \eqref{DB} of the unit disk 
\begin{equation*}
{\sf dist}_{{\sf har}}^{\DD}(w, w_0)
\leq {\sf dist}_{{\sf har}}^{\DD}(w, 0){\sf dist}_{{\sf har}}^{\DD}(0, w_0)
\leq \frac{1+|w|}{1-|w|}\frac{1+|w_0|}{1-|w_0|}=\Bigl(\frac{1+R}{1-R}\Bigr)^2
\end{equation*}
 and \eqref{1i} can be rewritten as
\begin{equation}\label{1i+}
b\ln \frac{1}{|w|}+v(w)
\leq  \Bigl(\frac{1+R}{1-R}\Bigr)^2\Bigl( b\ln \frac{1}{|w_0|}+v(w_0)\Bigr), \quad |w|=|w_0|.
\end{equation}
By virtue of identically equal $R$ modules of points $w$  and $w_0$ lying on the same circle due to the independence of the multiplier $t:=\Bigl(\frac{1+R}{1-R}\Bigr)^2$ from none of the variables for reasons of symmetry in relation to \eqref{1i+}, the following inequality is also true:
\begin{equation}\label{1i++}
\Bigl(\frac{1-R}{1+R}\Bigr)^2\Bigl(b\ln \frac{1}{|w|}+v(w)\Bigr)
\leq   b\ln \frac{1}{|w_0|}+v(w_0).
\end{equation}
Thus, for each positive  harmonic function $h$ on $\overline \DD\setminus 0$ of the form \eqref{hhh}
in view of  \eqref{1i+} and \eqref{1i++} we obtain inequalities
\begin{equation}\label{hard-}
\Bigl(\frac{1-R}{1+R}\Bigr)^2  h(w)\leq h(w_0) \leq \Bigl(\frac{1+R}{1-R}\Bigr)^2 h(w), \quad   
|w|=|w_0|=R\in (0,1)
\end{equation}
 Using the homothety centered at zero, these inequalities extend to all positive harmonic functions $h$ on 
on the punctured unit circle $\DD\setminus 0$, and by definition of the Harnack distance, inequalities \eqref{hard-} mean that
the Harnack distance between points $w$ and $w_0$ does not exceed $1+2R$.
Thus, according to \eqref{mainpr}  under condition $|w|=1$ we obtain
\begin{equation}\label{disty}
{\sf dist}_{{\sf har}}^{\DD\!\setminus\!0}\bigl(w,\partial (R\DD)\bigr) 
=\sup_{w_0\in \partial(R\DD)}{\sf dist}_{{\sf har}}^{\DD\!\setminus\!0}(w,w_0) 
\leq \Bigl(\frac{1+R}{1-R}\Bigr)^2
\end{equation}

At this stage, by Corollary \ref{cor1}, 
it is established that
for any function $h\colon [0,r]\to {\mathbb R}^+$, $0<r\leq 2R$, such that  $h(0)=0$ and 
\begin{equation}\label{Nr}
N_0^h(r)\overset{\eqref{N0h}}{:=}\int_0^r\frac{h(s)}{s}{\,{\mathrm d}} s<+\infty
\end{equation}
there is a set  $E\subset S$ such that
\begin{equation*}
 \inf_{z\in S\setminus E} u(x)\overset{\eqref{DB1}, \eqref{Bbbk},\eqref{disty}}{\geq} -\biggl(
\frac{2s_0}{R-s_0}+\frac{\ln (2R/r)}{-\ln R}
\Bigl(\frac{1+R}{1-R}\Bigr)^2+N_0^h(r)\biggr)
\sup_{\partial \DD} u;
\end{equation*}
and  the following estimate from above holds:
\begin{equation*}
\mathfrak m_h^r(E)\overset{\eqref{{eEr}E}}{\leq} 
\frac{5^{2}}{\Bbbk(R+l_D^G)-\Bbbk(R)}\sup\limits_{x\in \partial D}{\sf dist}_{{\sf har}}^{G\!\setminus\!o}\bigl(x,\partial B_o(R)\bigr)\overset{\eqref{disty}}{\leq} 
\frac{5^{ 2}}{-\ln R}\Bigl(\frac{1+R}{1-R}\Bigr)^2.
\end{equation*}
Finally,  choice  \eqref{Nsh} is obtained by directly calculating the integral from \eqref{N0}, and the estimate 
follows from \eqref{pBB}:
\begin{equation*}
p\text{-}{\mathfrak m}^{\text{\tiny $r$}}(E)\overset{\eqref{{eEr}E}}{\leq}
\frac{1}{B}\cdot\frac{5^{2}}{\ln(R+l_D^G)-\ln R}\sup\limits_{z\in \partial (R\DD)}{\sf dist}_{{\sf har}}^{\DD\!\setminus\!o}\bigl(z,\partial B_o(R)\bigr)
\overset{\eqref{rdd1},\eqref{disty}}{\leq} \frac{1}{B}\frac{5^2}{-\ln R}
\Bigl(\frac{1+R}{1-R}\Bigr)^2,
\end{equation*}
what gives the final \eqref{pB}.
\end{proof}
\begin{corollary} Let under the conditions \eqref{Ss0} 
 $u$ be  a subharmonic function on $\overline \DD$ with $u(0) = 0$. Choose $r:=2R$ (cf. with \eqref{rR})
 Then for any function $h\colon [0,2R]\to {\mathbb R}^+$ such that  $h(0)=0$ and 
(cf. with \eqref{Nr})
\begin{equation*}
N_0^h(2R):=\int_0^{2R}\frac{h(s)}{s}{\,{\mathrm d}} s<+\infty
\end{equation*}
there is a set  $E\subset S$ such that
\begin{equation}\label{frb-}
 \inf_{z\in S\setminus E} u(x)
\geq -\biggl(
\frac{2s_0}{R-s_0}
+N_0^h(2R)\biggr)
\sup_{\partial \DD} u,
\end{equation}
where the exceptional set $E$ is bounded in the sense \eqref{NB}  with $r:=2R$  
and with the same continuation \eqref{Nsh}--\eqref{pB} of the formulation of theorem\/ {\rm \ref{th2},} but everywhere with $r:=2R$.
\end{corollary}
\begin{proof} A significant simplification \eqref{frb-} of the estimate \eqref{frb-}  compared to  
\eqref{frb} related to the disappearance of the essential addition
$$
\frac{\ln (2R/r)}{-\ln R}
\Bigl(\frac{1+R}{1-R}\Bigr)^2
$$
is based on a trivial observation: if $r:=2R$, then  $\ln (2R/r)=0$.

\end{proof}

\end{document}